\documentclass{article}

\usepackage{amsmath,amssymb}
\usepackage{amsthm}
\usepackage{enumitem}
\usepackage[toc]{appendix}
\usepackage{multicol}
\usepackage{hyperref}
\usepackage{url}
\usepackage{graphicx}

\usepackage{color}

\usepackage{listings}
\usepackage{tikz,pgfplots}
\tikzset{
every node/.style={circle, inner sep=2pt}
}

\usetikzlibrary{matrix,arrows,calc}

\pgfplotsset{
    legend entry/.initial=,
    every axis plot post/.code={%
        \pgfkeysgetvalue{/pgfplots/legend entry}\tempValue
        \ifx\tempValue\empty
            \pgfkeysalso{/pgfplots/forget plot}%
        \else
            \expandafter\addlegendentry\expandafter{\tempValue}%
        \fi
    },
}

\newtheorem{theorem}{Theorem}
\newtheorem{lemma}[theorem]{Lemma}
\newtheorem{proposition}[theorem]{Proposition}
\newtheorem{corollary}[theorem]{Corollary}

\theoremstyle{definition}
\newtheorem{definition}[theorem]{Definition}
\newtheorem{observation}[theorem]{Observation}
\newtheorem{remark}[theorem]{Remark}
\newtheorem{example}[theorem]{Example}
\newtheorem{problem}{Problem}
\newtheorem{question}[theorem]{Question}

\def \mr {\operatorname{mr}}

\newcommand{\diag}{\operatorname{diag}}
\newcommand{\spec}{\operatorname{spec}}
\newcommand{\SNF}{\operatorname{SNF}}
\newcommand{\minors}{\operatorname{minors}}
\newcommand{\sign}{\operatorname{sign}}

\pgfplotsset{compat=1.14}
\begin{document}


\title{Eigenvalues, Smith normal form and determinantal ideals}

\author{Aida Abiad$^{a}$, Carlos A. Alfaro$^b$, Kristin Heysse$^c$ and Marcos C. Vargas$^b$
\\ \\
{\small $^a$Department of Mathematics: Analysis, Logic and Discrete Mathematics} \\
{\small Ghent University, Ghent, Belgium}\\
{\small {\tt
aida.abiad@ugent.be}} \\
{\small $^b$ Banco de M\'exico} \\
{\small Mexico City, Mexico}\\
{\small {\tt
\{carlos.alfaro,marcos.vargas\}@banxico.org.mx}} \\
{\small $^c$Department of Mathematics, Statistics, and Computer Science} \\
{\small Macalester College, USA}\\
{\small {\tt
kheysse@macalester.edu}}\\
}

\date{}
\maketitle






\maketitle

\begin{abstract}
	Determinantal ideals of graphs generalize, among others, the spectrum and the Smith normal form (SNF) of integer matrices associated to graphs.
	In this work we investigate the relationship of the spectrum and the SNF with the determinantal ideals. We show that an eigenvalue divides the $k${\it -th} invariant factor of its SNF if the eigenvalue belongs to a variety of the $k${\it -th} univariate integer determinantal ideal of the matrix. This result has as a corollary a theorem of Rushanan.
    We also study graphs having the same determinantal ideals with at  most  one  indeterminate; the socalled codeterminantal graphs, which generalize the concepts of cospectral and coinvariant graphs. We establish a necessary and  sufficient condition for graphs to  be codeterminantal on $\mathbb{R}[x]$, and some computational results on codeterminantal graphs 
    up to 9 vertices are presented.
    Finally, we show that complete graphs and star graphs are determined by the SNF of its distance Laplacian matrix.
    \\[7pt]
	
\end{abstract}

\section{Introduction}

Determinantal ideals are a central topic in both commutative algebra and algebraic geometry, and they also have numerous connections with invariant theory, representation theory, and combinatorics \cite{Mbook}. In this article we explore their connections with algebraic combinatorics. In particular, we investigate their relationships with the spectrum and the Smith normal form (SNF). 

As mentioned in \cite[Chapter~13.8.1]{bh}, there is no very direct connection between the SNF and the spectrum. However, a few papers trying to relate the spectrum and SNF of matrices associated to graphs have appeared in the literature. Rushanan \cite{R} studied the SNF and spectrum of non-singular matrices with integer entries. He established divisibility relations between the largest invariant factor $s_n$ and the product of all eigenvalues. Newman and Thompson \cite[Section~8]{nt} studied the relationship between eigenvalues and invariant factors of matrices over rings of algebraic integers. Their results are concerned with products of eigenvalues rather than individual eigenvalues (or subsets thereof). The connection between the eigenvalues and Smith form has also been studied by Kirkland \cite{kirkland} for integer matrices with integer eigenvalues arising from the Laplacian of graphs, and by Lorenzini \cite{lorenzini2008} for Laplacian matrices of rank $n-1$. Recently, Elsheikh and Giesbrecht \cite{eg} established some conditions under which the $p$-adic valuations of the invariant factors of an integer matrix are equal to the $p$-adic valuations of the eigenvalues.

In this article we investigate the relationship between the determinantal ideals and the SNF and the spectrum. Determinantal ideals of graphs, which can be viewed as a generalization of both the graph spectrum and the SNF, are ideals of minors of matrices whose entries are in a polynomial ring.
Let $M(G)$ be an $n\times n$ integer matrix associated to the graph $G$ with $n$ vertices.
There are many determinantal ideals that can be associated to $G$.
For example, if $x$ is an indeterminate, then the $k${\it -th} determinantal ideal is the ideal generated by the $k$-minors of $xI_n-M(G)$.. However, it is worth noting that this ideal has subtle differences depending whether it is included in $\mathbb{Z}[x]$ or in $\mathbb{R}[x]$; in $\mathbb{R}[x]$ the ideal is principal, and in $\mathbb{Z}[x]$ might not be principal and then we must compute its Gr\"obner bases to have a compact description of it. For the relevant background on the theory of determinantal ideals and rings, and their Gr\"obner bases, we refer to \cite{bcbook, EH}.

In this paper we extend a result by Rushanan \cite{R} that states that any eigenvalue of a diagonalizable matrix divides the last invariant factor of its SNF. We show that an eigenvalue divides the $k${\it -th} invariant factor of its SNF if the eigenvalue belongs to a variety of the $k${\it -th} univariate integer determinantal ideal of the matrix. 
Next, we investigate graphs having the same determinantal ideals with at  most  one  indeterminate, the socalled codeterminantal graphs, which generalize the concepts of cospectral and coinvariant graphs. We establish a necessary and  sufficient condition for graphs to  be codeterminantal on $\mathbb{R}[x]$, and we characterize when codeterminantal graphs are cospectral and coinvariant, respectively. Moreover, we present several computational results in which we look at codeterminantal graphs 
    up to 9 vertices.
    From this computational study, we observe that the best determinantal ideals to distinguish graphs are the univariate determinantal ideals in $\mathbb{Z}[x]$, since they provide a theory that unifies the spectrum and the SNF. We also look at the SNF and the spectrum of the adjacency, Laplacian, distance and distance Laplacian of all connected graphs up to 9 vertices, and from the numerical data we conclude that the SNF of the distance Laplacian matrix performs the best for distinguishing graphs. This extends the question of  van Dam and Haemers \cite{vDH} ``which graphs are determined by their spectrum?'' to the context of codeterminantal graphs. In this regard, in Section \ref{sec:graphsdeterminedbytheirdeterminantalideals}, we show that complete graphs and star graphs are determined by the SNF of its distance Laplacian matrix. Despite that univariate determinantal ideals in $\mathbb{Z}[x]$ are more difficult to compute, we observe that if a graph is determined by its spectrum, then it is determined by their univariate determinantal ideals in $\mathbb{Z}[x]$.
    Finally, we show that complete graphs and star graphs are determined by the SNF of its distance Laplacian matrix.

This article is structured as follows. We begin in Section \ref{sec:determinantalIdeals} by  establishing  some  basic  terminology  and  giving some basic properties of determinantal ideals. 
In Section \ref{sec:cospectral} we explore codeterminantal graphs, and we give the  results  of our exhaustive computational study in which we look at the determinantal ideals of all connected graphs up to 9 vertices. Finally, in Section \ref{sec:graphsdeterminedbytheirdeterminantalideals}, we show a few families of graphs that are determined by the SNF of the distance Laplacian matrix.

\section{Determinantal ideals}\label{sec:determinantalIdeals}
What is a determinantal ideal?
To answer this question we use \cite{Nbook} but we adopt the notation and terminology from \cite[Section 6.5.1]{EH}.

Let $\mathcal{R}$ be a commutative ring with unity, and consider a $n\times m$ matrix $M$ whose entries are in the polynomial ring $\mathcal{R}[X]$ with $X$ a set of $l$ indeterminates $x_1, \dots, x_l$.
We will assume that $n\leq m$ to simplify notation, because otherwise we can transpose the matrix without changing the determinants of its sub-matrices.
For $k\in [n]:=\{1, \dots, n\}$, let $\mathcal{I}=\{r_j\}_{j=1}^k$ and $\mathcal{J}=\{c_j\}_{j=1}^k$ be two sequences such that
\[
1\leq r_1 < r_2 < \cdots < r_k \leq n \text{ and } 1\leq c_1 < c_2 < \cdots < c_k \leq m.
\]
Let $M[\mathcal{I;J}]$ denote the submatrix of a matrix $M$ induced by the rows with indices in $\mathcal{I}$ and columns with indices in $\mathcal{J}$.
The determinant of $M[\mathcal{I;J}]$ is called a $k$-{\it minor} of $M$.
We denote by $\minors_k(M)$ the set of all $k$-minors of $M$.

\begin{definition}
For $k\in[n]$, the $k$-{\it th} \emph{determinantal ideal of a matrix}  $M$ with entries in $\mathcal{R}[X]$ (or just \emph{ideal}, if it is clear from the context), denoted $I_k(M)$, is the ideal generated by $\minors_k(M)$.
\end{definition}

Let $I\subseteq \mathcal{R}[X]$ be an ideal in $\mathcal{R}[X]$.
The \emph{variety} $V(I)$ of $I$ is defined as the set of common roots between polynomials in $I$.
In several contexts it will be more convenient to consider an extension $\mathcal{P}$ of $\mathcal{R}$ to define the variety $$V^\mathcal{P}(I):=\left\{ {\bf a}\in \mathcal{P}^l : f({\bf a}) = 0 \text{ for all } f\in I \right\}.$$

The following result, which we shall use in Section \ref{sec:charideals} when we study determinantal ideals with one variable, shows the contention between varieties.

\begin{proposition}\cite{Nbook}\label{prop:chaininclusionideals}
Let $M$ be an $n\times m$ matrix with entries in $\mathcal{R}[X]$.
Then, it holds that
\[
\langle 1\rangle \supseteq I_1(M) \supseteq \cdots \supseteq I_n(M) \supseteq \langle 0\rangle
\]
and
\[
\emptyset\subseteq V^\mathcal{P}(I_1(M)) \subseteq \cdots \subseteq V^\mathcal{P}(I_n(M)) \subseteq \mathcal{R}^l.
\]
\end{proposition}

An ideal is said to be {\it trivial} or {\it unit} if it is equal to $\langle1\rangle$ ($=\mathcal{R}[X]$). The {\it algebraic co-rank} $\gamma(M)$ of $M$ is the maximum integer $k$ for which $I_k(M)$ is trivial.

It is worth noting the use of a new notation to refer to the following property of the ideals of two matrices.
This is the main underlying concept that will allow us investigate codeterminantal graphs.

\begin{definition}\label{def:codeterminatalmatrix}
Let $M,N$ be two $n\times m$ matrices with entries in $\mathcal{R}[X]$.
We say that $M$ and $N$ are {\it codeterminantal} if $I_k(M)=I_k(N)$ for all $k\in[n]$.
\end{definition}

We also provide some background on codeterminantal graphs for our results later on.

\begin{lemma}\cite{Nbook}
Let $M,N$ be two matrices with entries in $\mathcal{R}[X]$.
Then, $I_k(MN)\subseteq I_k(M)\cap I_k(N).$
\end{lemma}

From which follows.

\begin{theorem}\cite{Nbook}\label{theo:equivalentcodeterminantal}
Let $M,N$ be two $n\times m$ matrices and suppose there exist $U,V,U',V'$ such that $M=U N V$ and $N = U' M V'$.
Then, $M$ and $N$ are codeterminantal.
\end{theorem}


\begin{definition}
A matrix $M$ is said to be {\it equivalent} to $N$, denoted by $M\sim N$, if there exist invertible matrices $U,V$ with entries in $\mathcal{R}[X]$ such that $M=UNV$.
\end{definition}

\begin{proposition}\cite{Nbook}\label{prop:equiimpliescodeterminantal}
If $M$ and $N$ are equivalent matrices, then 
$M$ and $N$ are codeterminantal.
\end{proposition}

A matrix $M$ can be transformed to $N$ by applying {\it elementary row and column operations}:
\begin{enumerate}
  \item interchanging any two rows or any two columns,
  \item adding integer multiples of one row/column to another row/column,
  \item multiplying any row/column by $\pm 1$.
\end{enumerate}
These operations are performed by multiplying $M$ by invertible matrices with entries in $\mathcal{R}[X]$.

\begin{corollary}\cite{Nbook}\label{coro:elementaryoperationscodeterminantal}
If $M$ is obtained from $N$ by means of elementary row and column operations, then $M$ and $N$ are codeterminantal.
\end{corollary}

In the particular case when the entries of the matrices are on a \emph{principal ideal domain} (PID), Proposition~\ref{prop:equiimpliescodeterminantal} can be improved.
Recall, for polynomial rings, $\mathcal{R}[x]$ is a PID if and only if $\mathcal{R}$ is a field.
It follows that we can not have a PID with more than one indeterminate.


\begin{proposition}\cite{Nbook}
Let $M,N$ be matrices with entries in a PID.
Then, $M$ and $N$ are equivalent if and only if $M$ and $N$ are codeterminantal.
\end{proposition}

\begin{theorem}\cite{jacobson}\label{teo:snf}
If $M$ is an $n\times m$ matrix of rank $r$ with entries in a PID, then $M$ is equivalent to a diagonal matrix
\[
diag(f_1(M),\dots,f_r(M),0,\dots,0),
\]
where $f_k(M)\neq0$ and $f_j(M)|f_k(M)$ for $j\leq k$.
\end{theorem}
The diagonal matrix obtained in Theorem~\ref{teo:snf} is known as {\it Smith normal form} (SNF) of $M$, and the elements in its diagonal are called {\it invariant factors}. The SNF of matrices over principal ideal domains such as $\mathbb{Z}$ and $\mathbb{Q}[x]$ have many applications in algebraic group theory, combinatorics, homology groups, integer programming, lattices, linear Diophantine equations, system theory, and analysis of cryptosystems \cite{cohen,kannan,schrijver,stanley}.

In the last 30 years, the SNF of integer matrices of graphs have been of great interest since it describes the Abelian group obtained from the cokernel. If we consider an $m\times n$ matrix $M$ with integer entries as a linear map $M:\mathbb{Z}^n\rightarrow \mathbb{Z}^m$, recall that the {\it cokernel} of $M$ is the quotient module $\mathbb{Z}^{m}/{\rm Im}\, M$.
This finitely generated Abelian group becomes a graph invariant when we take the matrix $M$ to be a matrix associated with the graph.
For instance, the cokernel of $A(G)$ is known as the {\it Smith group} of $G$ and is denoted $S(G)$, and the torsion part of the cokernel of $L(G)$ is known as the {\it critical group} $K(G)$ of $G$ (also known as {\it sandpile group}).
The structure theorem for finitely generated abelian groups implies the cokernel of $M$ can be described as:
$coker(M)\cong \mathbb{Z}_{f_1(M)} \oplus \mathbb{Z}_{f_2(M)} \oplus \cdots \oplus\mathbb{Z}_{f_{r}(M)} \oplus \mathbb{Z}^{m-r}$,
where $r$ is the rank of $M$ and $f_1(M), f_2(M), \ldots, f_{r}(M)$ are the \emph{invariant factors} of the SNF of the integer matrix $M$.
Much of the research done in this direction has been motivated by the sandpile group and its multiple relations with many other branches like algebraic geometry, hyperplane arrangements, parking functions to mention few.
We refer the reader interested in this topic to the book \cite{klivans}.

Little is known about Smith normal forms of distance matrices. In \cite{HW}, the Smith normal forms of the distance matrices were determined for trees, wheels, cycles, and complements of cycles and are partially determined for complete multipartite graphs. In \cite{BK}, the Smith normal form of the distance matrices of unicyclic graphs and of the wheel graph with trees attached to each
vertex were obtained.

A useful way to compute the invariant factors is given by the following result, which we shall use to prove our next result (Proposition \ref{ourresult1}).

\begin{theorem}[Elementary divisors theorem]\label{theo:edt}\cite{jacobson}
Let $M$ be an $n\times m$ matrix with entries in a PID.
Then the $k$-{\it th} invariant factor $f_k(M)$ of $M$ is equal to $\Delta_k(M)/ \Delta_{k-1}(M)$, where $\Delta_k(M)$ is the {\it greatest common divisor} of the $k$-minors of $M$ and $\Delta_0(M)=1$.
\end{theorem}

Note that when the ring is a PID, the $k${\it-th} determinantal ideal of $M$ is generated by $\Delta_k(M)$.
On the other hand, when the ring is not a PID, we will be interested in finding a minimal representation of the determinantal ideals like the obtained from Gr\"obner bases \cite{EH}.
Observe also that the worst case complexity of computing Gr\"obner bases in $\mathbb{Q}[x_1,\dots,x_l]$ is double exponential \cite{MaMe}, meanwhile computing the SNF of a (polynomial) matrix is performed in polynomial time \cite{kannan,kannan1}.

We are interested in studying the determinantal ideals of matrices associated to graphs with entries in polynomial rings over commutative rings or just commutative rings.

In this work we consider simple finite connected graphs. Let $G = (V,E)$ be a graph with vertex set $V=\{v_1, \dots, v_n\}$ and $E$ its edge set.
In the following, $\deg(G)$ denote the diagonal matrix containing the degrees of the vertices of $G$ in the diagonal. The {\it transmission} of a vertex $v_i$ of $G$ is the sum of the distances from $v_i$ to all other vertices, and we denote by $T(G)$ to the diagonal matrix of vertex transmissions.
In the following, $A(G)$ and $D(G)$ denote the {\it adjacency} and {\it distance} matrices of graph $G$ with $n$ vertices, respectively.
In this way, the {\it Laplacian} matrix $L(G)$ is equal to $\deg(G)-A(G)$ and the {\it distance Laplacian} matrix, denoted by $F(G)$\footnote{In the literature the notation $D^L$ is commonly used for the distance Laplacian matrix, but in the manuscript we have chosen $F$, instead, to simplify the notation.}, is $T(G)-D(G)$.

\begin{definition}\label{def:polynomialmat}
Given a graph $G$ with $n$ vertices, a set of indeterminates $X=\{x_u \, : \, u\in V(G)\}$ and a inderterminate $x$, we define the following polynomial matrices:
\begin{itemize}
    \item $A_X(G)=\diag(x_1, \dots,x_n)-A(G)$,
    \item $D_X(G)=\diag(x_1, \dots,x_n)-D(G)$,
    \item $A_x(G)=x I_n-A(G)$,
    \item $L_x(G)=xI_n-L(G)$,
    \item $D_x(G)=xI_n-D(G)$,
    \item $F_x(G)=xI_n-F(G)$,
\end{itemize}
\end{definition}
where $I_n$ denotes the identity matrix of size $n\times n$.

The determinantal ideals associated to the matrices $A_X$ and $D_X$ have been studied in \cite{alflin,at,AV,AV1,AVV,corrval} under the name of {\it critical ideals} and {\it distance ideals}, respectively.
On the other hand, the univariate determinantal ideals of $A_x$ are known as {\it characteristic ideals}.
The first reference to a characteristic ideal might be found in \cite{mccoy} dealing with the $n${\it -th} determinantal ideal of matrices of the form $M_x:=xI_n-M$.
We follow \cite{mccoy} to name {\it Laplacian characteristic ideals}, {\it distance characteristic ideals} and {\it distance Laplacian characteristic ideals} to the univariate determinantal ideals of the matrices $L_x$, $D_x$ and $F_x$, respectively.

Let us include more notation in order to clearly know the (polynomial) ring in which the determinantal ideal is defined; for that we shall use an additional superscript in the ideal.
Let $M(G)$ be one of the integer matrix associated to a graph $G$ introduced above, we will denote by $I^\mathcal{R}_k(M_X(G))$, the $k${\it -th} determinantal ideal of the matrix $M_X(G)=\diag(x_1,\dots,x_n)-M(G)$, which is contained in the polynomial ring $\mathcal{R}[X]$.
In the univariate case we will use $I^\mathcal{R}_k(M_x(G))\subseteq \mathcal{R}[x]$, and when the matrix has no indeterminate we use $I^\mathcal{R}_k(M(G))\subseteq \mathcal{R}$.
Similarly, $\gamma_\mathcal{R}(M_X)$ denote the number of trivial determinantal ideals under $\mathcal{R}[X]$.
We will be mainly interested when $\mathcal{R}$ is either $\mathbb{R}$, $\mathbb{Q}$ or $\mathbb{Z}$.
Subtle differences appear; for example, the ideal $\left<2,x\right>$ is not principal in $\mathbb{Z}[x]$, but is trivial in $\mathbb{Q}[x]$.

The next result is useful  when the ideal is defined over a PID. 

\begin{proposition}\label{ourresult1}
Let $\mathcal{R}[x]$ be a PID.
Let $M(G)$ be either of the adjacency, Laplacian, distance or distance Laplacian matrix of $G$. Then, the $k$-determinental ideal $I^\mathcal{R}_k\left(M_x(G)\right)$ of the graph $G$ is isomorphic to the $k$-determinental ideal of the Smith normal form of $M_x(G)$.
Moreover, $I^\mathcal{R}_k\left(M_x(G)\right)$ is generated by $\Delta_k\left(M_x(G)\right)$, the gcd of the $k${\it-th} minors of $M_x(G)$.
\end{proposition}
\begin{proof}
 If $M$ is a matrix with entries in a PID $\mathcal{R}[x]$, then $M$ is equivalent to its SNF, a diagonal matrix $\text{diag}(d_1\ldots, d_r, 0, \ldots, 0)$ of rank $r$. And $d_k = \Delta_k / \Delta_{k-1}$, where $\Delta_k$ is the gcd of the $k$-minors of $M$, see Theorem~\ref{theo:edt}. When restrict our studies to a PID, say a ring of polynomials with coefficients in reals, the $k$-th determinental ideal is principal, that is, it is generated by only one element: $\Delta_k$.
\end{proof}

By the above result, we can apply elementary operations in the matrix without changing the determinantal ideals in order to simplify the matrix and obtain the generator; this is easier than to calculate the Gr\"obner basis of all $k$-minors.

Let $M$ be a $n\times n$ integer matrix and let $M_X$ be the polynomial matrix $\diag(x_1,\dots,x_n)-M$.
We can recover the invariant factors of $M$ from the determinantal ideals of $M_X$ as described in the following result.
\begin{proposition}\label{prop:evalmultiplevariables}
Let $\mathcal{R}$ be a PID, $M$ an $n\times n$ matrix with entries in $\mathcal{R}$, ${\bf c}$ a row vector in $\mathcal{R}^n$ and $X=\{x_1,\dots,x_n\}$ a set of indeterminates.
Let $M_X=\diag(x_1,\dots,x_n)-M$.
Then the ideal $I^\mathcal{R}_k(M_X)$ evaluated at $X={\bf c}$ is generated by $\Delta_k({\bf c}I_n-M)$, the {\it gcd} of the $k$-minors of ${\bf c}I_n-M$ over $\mathcal{R}$.
\end{proposition}
\begin{proof}
When the determinantal ideal $I^\mathcal{R}_k(M_X)$ is evaluated at $X={\bf c}$, the ideal obtained is the generated by the $k$-minors of ${\bf c}I_n-M$.
Since $\mathcal{R}$ is a PID, then the evaluated ideal is principal and generated by the gcd of the $k$-minors of $M$.
\end{proof}

Note that the determinantal ideals of a (polynomial) matrix $M$ and its negative $-M$ are the same.
Thus from the determinantal ideals of matrix $A_X(G)$, we can recover the SNF of the adjacency matrix of $G$.
We can obtain many useful variants of the previous result, for example, by evaluating the determinantal ideals of $A_X(G)$ at $X=\deg(G)$, we can recover the SNF of the Laplacian matrix.

An application of Proposition \ref{prop:evalmultiplevariables} will be shown in Section \ref{sec:graphsdeterminedbytheirdeterminantalideals} to prove that some families of graphs are determined by the SNF of
the distance Laplacian matrix.

The next corollary shows one particular case which is useful for univariate determinantal ideals and when the diagonal of the matrix is constant. We shall use it in Section \ref{sec:charideals}.

\begin{corollary}\label{coro:evalunideterminantal}\label{theo:determinatalevaluatedsingleind}
Let $\mathcal{R}$ be a PID, $M$ an $n\times n$ matrix with entries in $\mathcal{R}$ and let $c\in\mathcal{R}$.
Let $M_x=xI_n-M$.
Then the ideal $I^\mathcal{R}_k(M_x)$ evaluated at $x=c$ is generated by $\Delta_k(cI_n-M)$, the {\it gcd} of the $k$-minors of $M$ over $\mathcal{R}$.
\end{corollary}

This is more restrictive, from the determinantal ideals of $A_x(G)$ we can recover the SNF of $A(G)$, and the $G$ is $r$-regular, by evaluating the determinantal ideals at $x=r$, we recover the SNF of $L(G)$.
This kind of evaluations will be used in the rest of the paper to make connections between determinantal ideals (in $\mathbb{Z}[X]$ or $\mathbb{Z}[x]$) and SNFs.

In the following subsections we focus on the determinantal ideals of the polynomial matrices mentioned in Definition \ref{def:polynomialmat}.

\subsection{Critical ideals} 

The determinantal ideals $I_k^\mathcal{R}(A_X(G))$ were defined and studied in \cite{corrval} as a generalization of the critical group, which is the torsion part of the cokernel of the Laplacian matrix.

\begin{example}\label{example:grobner1}
Consider the cycle with 4 vertices. Then
\[
A_X(C_4)=
\begin{bmatrix}
    x_0 & -1 &  0 & -1\\
    -1 & x_1 & -1 &  0\\
     0 & -1 & x_2 & -1\\
    -1 &  0 & -1 & x_3\\
\end{bmatrix}.
\]
Below we give the Gr\"obner bases of the critical ideals over $\mathbb{Z}[X_{C_4}]$:
\[
I^{\mathbb{Z}}_k(A_X(C_4))=
\begin{cases}
    \langle 1\rangle & \text{if } i \leq 2,\\
    \langle x_0 + x_2, x_1 + x_3, x_2x_3\rangle & \text{if } i =3,\\
    \langle x_0x_1x_2x_3 - x_0x_1 - x_0x_3 - x_1x_2 - x_2x_3\rangle & \text{if } i =4.\\
\end{cases}
\]
\end{example}

Critical ideals generalize the SNF of several matrices associated to graphs like the Laplacian, signless Laplacian and adjacency matrices of graphs, see references  \cite{AV,corrval}.
Thus, objects like the critical group and the Smith group can be recovered from critical ideals.
The following example illustrates it.
\begin{example}
Consider again Example~\ref{example:grobner1}.
By evaluating the critical ideals over $\mathbb{Z}[X]$ of $C_4$ at $X_{C_4}=\deg(C_4)=(2,2,2,2)$, we obtain that the gcd of the $k$-minors of $L(G)$ are $\Delta_k(L(C_4))=1$ for $k\leq 2$,  $\Delta_3(L(C_4))=4$ and  $\Delta_4(L(C_4))=0$.
Thus the critical group $K(C_4)\cong \mathbb{Z}_4$.
An evaluation of the critical ideals of $C_4$ at the zero vector gives us that $\Delta_k(A(C_4))=1$ for $k\leq 2$, and  $\Delta_k(A(C_4))=0$ for $k\in\{3,4\}$.
Therefore, the Smith group $S(C_4)\cong \mathbb{Z}^2$.
\end{example}

Note that the varieties of critical ideals can also be regarded as a generalization of the Laplacian and adjacency spectra of $G$.

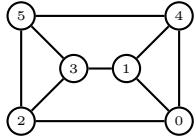
\begin{figure}[ht]
    \centering
    \begin{tikzpicture}[scale=.7,thick]
	\tikzstyle{every node}=[minimum width=0pt, inner sep=2pt, circle]
	\draw (1,0) node[draw] (0) {\tiny 0};
	\draw (0,1) node[draw] (1) {\tiny 1};
	\draw (-2,0) node[draw] (2) {\tiny 2};
	\draw (-1,1) node[draw] (3) {\tiny 3};
	\draw (1,2) node[draw] (4) {\tiny 4};
	\draw (-2,2) node[draw] (5) {\tiny 5};
	\draw  (0) edge (1);
	\draw  (0) edge (2);
	\draw  (0) edge (4);
	\draw  (1) edge (3);
	\draw  (1) edge (4);
	\draw  (2) edge (3);
	\draw  (2) edge (5);
	\draw  (3) edge (5);
	\draw  (4) edge (5);
    \end{tikzpicture}
    \caption{A graph whose fourth critical ideal is special.}
    \label{fig:varietiesofcriticalideals}
\end{figure}

Not much is known on the varieties of the critical and distance ideals.
An interesting example from \cite{alflin} is that a Gr\"obner base of the fourth critical ideal of the graph $G$ in Figure~\ref{fig:varietiesofcriticalideals} is given by
$I_4^\mathbb{Z}(A_X(G))=\langle x_0 + x_5 - 1, x_1 + x_5 - 1, x_2 - x_5, x_3 - x_5, x_4 + x_5 - 1, x_5^2 - x_5 - 1\rangle$.
Here, a quadratic polynomial is one of the generators of the ideal, which could allow the existence of complex solutions in the varieties of critical ideals (in $\mathbb{Z}[X]$) of undirected graphs.
On the other hand, in \cite{alflin} it was conjectured that the minimum rank $\mr_\mathbb{R}(G)\leq\gamma_\mathbb{R}(A_X(G))$.
This conjecture is related with the varieties of the critical ideals since if for $k=\gamma_\mathbb{R}(A_X(G))$, the variety $V\left(I_{k+1}^\mathbb{R}(A_X(G))\right)$ is not empty and is contained in $\mathbb{R}$, then $\mr_\mathbb{R}(G)\leq\gamma_\mathbb{R}(A_X(G))$.
This conjecture is known to be true \cite{alflin} for all graphs with at most 7 vertices, and \cite{alfaro} for all graphs with minimum rank at most 3.

One property of critical ideals is that they are monotone induced \cite{corrval}, that is, if $H$ is an induced subgraph of $G$, then, for each $k\in\left[|V(H)|\right]$, the $k${\it -th} critical ideal of $H$ is included in the $k${\it -th} critical ideal of $G$.
This behaviour is not true in general for the critical group, for example the critical group of $K_4$ is $\mathbb{Z}_4^2$ meanwhile the critical group of $K_5$ is $\mathbb{Z}_5^3$.
This property of the critical ideals can be used \cite{AV,AV1,AVV} to find a characterization of $\mathcal K_k$, the graphs whose critical group have $k$ invariant factors equal to 1.
The first result in this direction appeared when D. Lorenzini and, independently, A. Vince noticed in \cite{lorenzini1991,vince} that the graphs in $\mathcal K_1$, the graphs having critical group with one invariant factor equal to 1, consist only of the complete graphs. It is still an open problem to characterize graphs in $\mathcal{K}_k$ \cite{merino}.
A complete characterization of $\mathcal K_2$ was obtained in \cite{AV} using critical ideals defined over $\mathbb{Z}[X]$.
However, the characterization of the graphs in $\mathcal K_3$ seems to be a hard problem \cite{AV1}.
For digraphs case, the characterization of digraphs with at most 1 invariant factor equal to 1 was completely obtained in \cite{AVV}; this characterization turned out to be the same for digraphs with minimum rank equal to 1 and for digraphs with zero-forcing number equal to $n-1$. 

Let $\sigma$ be a permutation on $V(G)$.
Then $\sigma G$ is a graph on $V(G)$ such that $\{i,j\}\in E(G)$ if and only if $\{ \sigma(i),\sigma(j)\}\in E(\sigma G)$.
Two graphs $G$ and $G'$ on the same vertex set $V$ are called $n$-{\it cospectral} if there exists a permutation $\sigma$ on $V$ such that $\det(A_X(G))=\det(A_{\sigma X}(\sigma G'))$.

\begin{proposition}\cite[Proposition 1]{grt}\label{prop:ncospectral}
	Let $G$ and $G'$ be two graphs with $n$ vertices.
	Then $G$ and $G'$ are isomorphic if and only if they are $n$-cospectral.
\end{proposition}

The proof of the above result is a consequence of a bijection between the edges of $G$ and the monomials of degree $n-2$ in $\det(A_X(G))$ given by $\{i,j\} \mapsto -\prod_{k\neq i,j}x_k$.
We note that this bijection is a generalization of the fact that the coefficient of the term $x^{n-2}$ in $\det(A_x(G))$ is the negative of the number of edges of $G$.
Since the determinant of the matrix is equal to the generator of the $n${\it -th} critical ideal, then we obtain the following corollary.

\begin{corollary}
	Let $G$ and $G'$ be two graphs with $n$ vertices.
	Then $G$ and $G'$ are isomorphic if and only if there exists a permutation $\sigma$ on $V(G')$ such that the $n${\it -th} critical ideals of $G$ and $\sigma G'$ are equal.
\end{corollary}

\subsection{Distance ideals}
The determinantal ideals $I_k^\mathcal{R}(D_X(G))$ were previously studied in \cite{at}. The next result extends Proposition~\ref{prop:ncospectral} in \cite{grt} to the matrix $D_X$.

\begin{proposition}\label{prop:distanceideals}
	Let $G$ and $G'$ be two graphs with $n$ vertices. Then $G$ and $G'$ are isomorphic if and only if there exists a permutation $\sigma$ on $V$ such that $\det(D_X(G))=\det(D_{\sigma X}(\sigma G'))$.
\end{proposition}

\begin{proof}
Let $\sigma$  be a permutation on $V(G)$.
Then $\sigma G$ is a graph on $V(G)$ such that $\{i,j\}\in E(G)$ if and only if $\{ \sigma(i),\sigma(j)\}\in E(\sigma G)$.
Since $$\det(M)=\sum_{\sigma\in S_n} \sign (\sigma)\prod_{i=1}^n M_{i,\sigma(i)},$$ there is a bijection between the edges of $G$ and the monomials of degree $n-2$ in $\det(D(G,X_G))$ given by $\{i,j\} \mapsto -d_G(i,j)^2\prod_{k\neq i,j}x_k$.
Then the desired result follows.
\end{proof}

Proposition \ref{prop:distanceideals} implies that there are not two graphs with the same distance ideals.

\begin{corollary}
	Let $G$ and $G'$ be two graphs with $n$ vertices.
	Then $G$ and $G'$ are isomorphic if and only if there exists a permutation $\sigma$ on $V$ such that $I_n^\mathcal{R}(D_X(G))=I_n^\mathcal{R}(D_{\sigma X}(\sigma G))$.
\end{corollary}

\subsection{Univariate determinantal ideals}\label{sec:charideals}


Our main result of this section is a relation between the SNF of an integer matrix $M$ and its eigenvalues. This can be seen as an extension of a result by Rushanan (see Theorem 1 in \cite{R}), who  studies  the  relationship  between  the  spectrum  and  the SNF  of  non-singular  integer  matrices  with  integer  eigenvalues. Note that his result is  valid  for any PID.

\begin{theorem}\label{thm:spectrimSNF}
Let $M$ be a $n\times n$ symmetric integer matrix and let $M_x=xI_n-M$.
Suppose $\lambda$ is a common root of the polynomials in $I_k^\mathbb{Z}(M_x)$, then $\lambda$ is an eigenvalue of $M$,
$\lambda$ is a real number,
and $\lambda$ is a factor of $\Delta_k(M)$ with division defined in
the ring of algebraic integers.
\end{theorem}
\begin{proof}
The polynomial ring $\mathbb{Z}[x]$ is Notherian, therefore we can assume the determinantal ideal $I_k^\mathbb{Z}(M_x)$ is finitely generated by the non-constant polynomials\\ $q_1(x),\dots,q_l(x)$, and $\lambda$ is a common root of these polynomials. By the contention of the varieties shown in Proposition~\ref{prop:chaininclusionideals} and the fact that $M_x$ is a symmetric matrix, it follows that $\lambda$ is a real root of the characteristic polynomial $\det(M_x)$. Therefore $\lambda$ is an eigenvalue of $M$. 
That is, the polynomial $x-\lambda$ is a factor of each polynomial in the determinantal ideal $I_k^\mathbb{Z}(M_x)$.
On the other hand, when $I_k^\mathbb{Z}(M_x)$ is evaluated at $x=0$, then by Corollary~\ref{theo:determinatalevaluatedsingleind} it follows that $I_k^\mathbb{Z}(M)$ is generated by $\Delta_k(M)$.
Therefore, $\lambda$ is a factor of $\Delta_k(M)$, completing the proof.
\end{proof}

As a corollary of Theorem \ref{thm:spectrimSNF}, we obtain Rushanan's result \cite{R}.

\begin{corollary}\cite{R}
Let $\lambda$ be an eigenvalue of an integer nonsingular symmetric matrix $M$, and let $d_n$ be the last invariant factor of the SNF of $M$. Then $\lambda|s_n$ (division defined in the ring of algebraic integers).
\end{corollary}
\begin{proof}
Let $d_1,\dots,d_n$ be the invariant factors of the SNF of $M$. 
Let $m$ be the multiplicity of $\lambda$.
Then $(x-\lambda)^m$ divides $\det(xI-M)$ that equals $\Delta_n(xI-M)$.
Suppose $\lambda$ does not divide $d_n=\Delta_n(M)/\Delta_{n-1}(M)$.
This implies that $(x-\lambda)^m$ also divides $\Delta_{n-1}(M)$.
But since $d_{n-1}$ divides $d_n$, then $(x-\lambda)^m$ is a factor of $\Delta_{n-2}(xI-M)$.
Analogously, for any $k$, $(x-\lambda)^m$ is a factor of $\Delta_k(xI-M)$.
In particular the gcd of all entries of $xI-M$ is equal to $(x-\lambda)^m$.
But this only is possible when all non-diagonal entries of $xI-M$ are zero and $m=1$.
From which follows $M=\diag(\lambda,\dots,\lambda)$ is already in its SNF.
Which is a contradiction.
\end{proof}

Theorem \ref{thm:spectrimSNF} shows that the varieties of the univariate determinantal ideal of an undirected graph are contained in the reals, therefore, from here on, we will denote the variety of $I_k^\mathcal{R}(M_x)$ by $V^\mathbb{R}$ or $V$ when we consider $\mathcal{R}=\mathbb{R}$ or $\mathbb{Z}$ and $M_x=xI_n-M$ with $M$ being a symmetric integer matrix.


\begin{example}
Consider the bipartite graph $K_{3,3}$. Then
\[
L_x(K_{3,3})=xI_6-L(K_{3,3})=
\begin{bmatrix}
x - 3 &     0 &     0 &     1 &     1 &     1\\
    0 & x - 3 &     0 &     1 &     1 &     1\\
    0 &     0 & x - 3 &     1 &     1 &     1\\
    1 &     1 &     1 & x - 3 &     0 &     0\\
    1 &     1 &     1 &     0 & x - 3 &     0\\
    1 &     1 &     1 &     0 &     0 & x - 3\\
\end{bmatrix}.
\]
The Gr\"obner bases of the $k${\it-th} Laplacian characteristic ideal $I_k^\mathbb{Z}(L_x(K_{3,3}))$ are: 
\[
\begin{cases}
\langle1\rangle & \text{when }k=1,2,\\
\langle x-3\rangle & \text{when }k=3,\\
\langle (x-3)^2\rangle & \text{when }k=4,\\
\langle (x-3)^3(x+9), 3(x - 3)^3\rangle & \text{when }k=5,\\
\langle x (x - 3)^4 (x - 6)\rangle & \text{when }k=6.\\
\end{cases}
\]
By evaluating each $I_k^\mathbb{Z}(L_x(K_{3,3}))$ at $x=0$, we recover the SNF of the Laplacian matrix of $K_{3,3}$: $\diag(1,1,3,3,9,0)$, and from the last ideal we obtain that the eigenvalues of the Laplacian matrix are $\{6, 0, 3, 3, 3, 3\}$.
\end{example}

The following result states that the varieties of the univariate determinantal ideals are the same whether they belong to $\mathbb{Z}[x]$ or $\mathbb{R}[x]$, despite that the ideals could be different.

\begin{proposition}\label{prop:equalvarietiesunivariate}
For any $n\times n$ symmetric integer matrix $M$, let $M_x=xI_n-M$.
Then $V^\mathbb{R}(I_k^\mathbb{Z}(M_x))= V^\mathbb{R}(I_k^\mathbb{R}(M_x))$. 
\end{proposition}
\begin{proof}
Suppose that the determinantal ideal $I_k^\mathbb{Z}(M_x)$ is generated by the non-constant polynomials $q_1(x),\dots,q_l(x)$.
Since $\mathbb{R}[x]$ is a PID, then $\left\langle q_1(x),\dots,q_l(x)\right\rangle\subseteq\mathbb{R}[x]$ is generated by a unique polynomial, say $q(x)$, that is the generator of $I_k^\mathbb{R}(M_x)$.
Therefore, there exist polynomials $p_1(x),\dots,p_l(x)\in\mathbb{R}[x]$ such that $p_1(x)q_1(x)+\dots+p_l(x)q_l(x)=q(x)$.
Then, if $\lambda$ is a common root of $q_1(x),\dots,q_l(x)$, then $\lambda$ is a root of $q(x)$.
On the other hand, if $\lambda$ is a root of $q(x)$, then $\lambda$ is a common root of all polynomials in $I_k^\mathbb{R}(M_x)$, in particular, $\lambda$ is root of the polynomials $q_1(x),\dots,q_l(x)$, which are the generators of $I_k^\mathbb{Z}(M_x)$.
\end{proof}

The varieties of the univariate determinantal ideals of the matrices $A_x$ and $L_x$ can be also used to bound the minimum rank and the zero-forcing number of a graph; this particular application of the critical ideals appeared in \cite{alflin}.
However, critical ideals are finer invariants than the determinantal ideals of the matrices $A_x$ and $L_x$.
On the other hand, the evaluation of the determinantal ideals with entries in $\mathbb{Z}[x]$ of $A_x$ at $x=0$ gives us the structure of the Smith group, and if the graph is $r$-regular the evaluation of the ideals at $x=r$ gives us the structure of the sandpile group.
The following two results provide an illustration of the above and they follow by applying Proposition 3.15 and Theorem 3.16 in \cite{corrval} to the complete graph.




\begin{proposition}
    The characteristic ideals of the complete graph $K_n$ with $n$ vertices are given by
    \[
    I^\mathbb{Z}_k(A_x(K_n))=
    \begin{cases}
        \langle  (x+1)^{k-1} \rangle & \text{if } k \leq n-1,\\
        \langle (x + 1-n)(x+1)^{n-1} \rangle & \text{if } k=n,\\
    \end{cases}
    \]
    and
    \[
    V\left(I^\mathbb{Z}_k(A_x(K_n))\right)=
    \begin{cases}
        \emptyset & \text{if } k=1,\\
        \{-1\} & \text{if } 2\leq k \leq n-1,\\
        \{n-1,-1\} & \text{if } k=n.\\
    \end{cases}
    \]
\end{proposition}
From the above result, we deduce the following corollary.

\begin{corollary}
Let $K_n$ be the complete graph with $n$ vertices. Then
\begin{description}
    \item[$(i)$] the minimum rank $\mr(K_n)$ of $K_n$ is at most $1$,
    \item[$(ii)$] the sandpile group $K(K_n)$ is isomorphic to $\mathbb{Z}_{n}^{n-2}$,
    \item[$(iii)$] the Smith group $(K_n)$ is isomorphic to $\mathbb{Z}_{n-1}$.
\end{description}
\end{corollary}
\begin{proof}
$(i)$ Since $V\left(I^\mathbb{Z}_2(A_x(K_n))\right)$ is not empty, we can apply the result in \cite{alflin} which shows that if there exists ${\bf a}\in\mathbb{R}^n$ such that $I^{\mathbb{R}}_{k}(A_X(G))\mid_{X={\bf a}}=\langle 0\rangle$ for some $k$, then $\mr(G)\leq k-1$, and we obtain $\mr(K_n)\leq1$.

$(ii)$ By evaluating the characteristic ideals at $t=n-1$, we obtain $\Delta_k(L(K_n))=n^{k-1}$ for $1\leq k\leq n-1$, and $\Delta_n(L(K_n))=0$, from which follows that the SNF of the Laplacian matrix of $K_n$ is $\diag(1,n,\dots,n,0)$.

$(iii)$ By evaluating the characteristic ideals at $t=0$, we obtain $\Delta_k(A(K_n))=1$ for $1\leq k\leq n-1$, and $\Delta_n(A(K_n))=n-1$, from which follows that the SNF of the adjacency matrix of $K_n$ is $\diag(1,\dots,1,n-1)$.
\end{proof}


\subsection{Overview of Section \ref{sec:determinantalIdeals}}

Figure~\ref{fig:mindmap} provides an overview of the determinantal ideals which have been investigated in Section \ref{sec:determinantalIdeals}.

\begin{figure}[ht]
    \centering
    \begin{tikzpicture}[fill=gray,scale=0.8]
\draw (1.7,0) node {Thm.~\ref{thm:spectrimSNF}};
\draw[double] (-2,-1.5) rectangle (2.8,0.5);
\draw (-0.9,0) node {
SNF(M)
};
\draw[dashed,thick] (0.7,-1.9) rectangle (5.5,1.9);
\draw (2.2,1.3) node {
spectrum(M)
};
\draw[dotted,thick] (0.3,-2.5) rectangle (7,3.5);
\draw (2,2.8) node {$\left\{I_k^\mathbb{R}(M_X)\right\}_{k=1}^n$};
\draw[thick] (-5,-2.2) rectangle (5.9,2.2);
\draw (-3.5,1.5) node {$\left\{I_k^\mathbb{Z}(M_x)\right\}_{k=1}^n$};
\draw (-5.5,-2.8) rectangle (8,4);
\draw (-3.8,3.3) node {$\left\{I_k^\mathbb{Z}(M_X)\right\}_{k=1}^n$};
\end{tikzpicture}
    \caption{Relations between different determinantal ideals of an integer matrix $M$.}
    \label{fig:mindmap}
\end{figure}
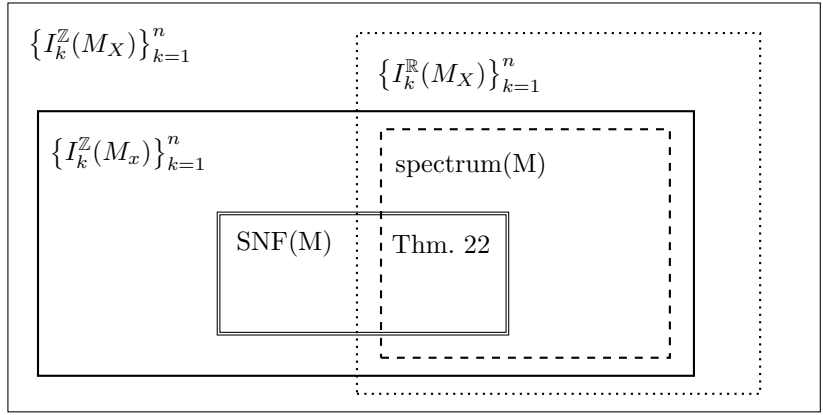

Keeping using the notation above, for any $n\times n$ integer matrix $M$, we denote $M_x=xI_n-M$ and $M_X=\diag(x_1,\dots,x_n)-M$.
From the similarity theory of matrices and the fact that each determinantal ideal $I_k^\mathbb{R}(M_x)$ is principal, we know that determinantal ideals $\left\{I_k^\mathbb{R}(M_x)\right\}_{k=1}^n$, $\text{spectrum}(M)$ and $\text{SNF}(M_x)$ are equivalent.
These concept are enclosed in a unique dashed region.
The SNF of $M$ and determinantal ideals $\left\{I_k^\mathbb{Z}(M_x)\right\}_{k=1}^n$ are equivalent, and are enclosed in the double lined region.
The intersection of the SNF of $M$ and its spectrum is given by Theorem \ref{thm:spectrimSNF}. Both the SNF and the spectrum are contained in the determinantal ideals $\left\{I_k^\mathbb{Z}(M_x)\right\}_{k=1}^n$ that is enclosed in the bold line region. 
On the other hand, the spectrum is generalized by the varieties of the determinantal ideals $\left\{I_k^\mathbb{R}(M_X)\right\}_{k=1}^n$, which are enclosed in the dotted region.
Finally, observe that all concepts lie in the determinantal ideals $\left\{I_k^\mathbb{Z}(M_X)\right\}_{k=1}^n$ which are enclosed in the biggest rectangle.

Given a graph $G=(V,E)$ with $n$ vertices, consider the $n\times n$ matrix $X(G)$ with rows and columns indexed by the vertices of $G$, in which the $(u,v)$-entry is the indeterminate $x_{u,v}$ if $uv\in E(G)$, and $0$ otherwise.
It will be interesting to study the determinantal ideals of this matrix, since they generalize the determinantal ideals presented in Section~\ref{sec:determinantalIdeals}.
The work of Katzman \cite{katzman} can be regarded as the determinantal ideals of $X(G)$ with $G$ a complete digraph with loops.
This research direction can also be linked with the minimum rank and determinantal varieties in algebraic and tropical geometry \cite{ms}.

\section{Codeterminantal  graphs}\label{sec:cospectral}

Codeterminantal graphs have already been defined in Section \ref{sec:determinantalIdeals}. In fact, the following definition is equivalent to Definition~\ref{def:codeterminatalmatrix}, but here we need to introduce a slightly different notation in order to specify the matrix and the ring in which we are working on.

\begin{definition}
Let $M$ be either of the adjacency, Laplacian, distance or distance Laplacian matrices. Two graphs $G$ and $H$ are  $M_x^\mathcal{R}$-{\it codeterminantal} if $I^\mathcal{R}_k(M_x(G))=I^\mathcal{R}_k(M_x(H))$ for each $k\in[n]$. We say that $G$ and $H$ are $M_x^\mathcal{R}$-{\it codeterminantal mates} if $G$ and $H$ are $M_x^\mathcal{R}$-codeterminantal.
\end{definition}

In Section \ref{sec:determinantalIdeals} we showed that there exists no pair of codeterminantal graphs with respect the critical ideals or the distance ideals.
In this section, we shall explore the notion of codeterminantal for matrices associated to graphs with at most one indeterminate.

In particular, when the matrix has no indeterminate, we have that two graphs $G$ and $H$ are $M^\mathcal{R}$-{\it codeterminantal} if $I^\mathcal{R}_k(M(G))=I^\mathcal{R}_k(M(H))$ for each $k\in[n]$. Thus $G$ and $H$ are $M^\mathcal{R}$-{\it codeterminantal mates} if $G$ and $H$ are $M^\mathcal{R}$-codeterminantal.
When $\mathcal{R}=\mathbb{Z}$, the ideal $I^\mathcal{R}_k(M(G))$ is generated by $\Delta_k(M)=\gcd(\minors_k(M))$.
This setting will also be explored. 

Returning to the univariate case, the following example illustrates that the there are subtle differences in taking different coefficients of the polynomial ring.

\begin{example}\label{example:differencechangingcoefficientspolynomialring}
	Let $G_1$ and $G_2$ be the graphs shown in Fig.~\ref{Fig:CIAndIF}.

\begin{figure}[ht]
	{
	\begin{center}
	\begin{tabular}{c@{\extracolsep{3cm}}c}
		\begin{tikzpicture}[scale=.7]
			\tikzstyle{every node}=[minimum width=0pt, inner sep=2pt, circle,thick]
			\tikzstyle{every path}=[thick]
			\draw (180:2) node[draw] (v1) {};
 			\draw (0:1) node[draw] (v5) {};
 			\draw (90:1) node[draw] (v3) {};
 			\draw (180:1) node[draw] (v2) {};
 			\draw (270:1) node[draw] (v4) {};
			\draw (0:2) node[draw] (v6) {};
			\draw (v1) to (v2);
 			\draw (v2) to (v3);
 			\draw (v2) to (v4);
 			\draw (v3) to (v5);
 			\draw (v4) to (v5);
			\draw (v3) to (v4);
			\draw (v5) to (v6);
		\end{tikzpicture}
		&
		\begin{tikzpicture}[scale=.7]
			\tikzstyle{every node}=[minimum width=0pt, inner sep=2pt, circle,thick]
			\tikzstyle{every path}=[thick]
			\draw (-1,-1) node[draw] (v1) {};
 			\draw (1,1) node[draw] (v5) {};
 			\draw (0,0) node[draw] (v3) {};
 			\draw (-1,1) node[draw] (v2) {};
 			\draw (0,-1) node[draw] (v4) {};
			\draw (1,-1) node[draw] (v6) {};
			\draw (v1) to (v2);
 			\draw (v1) to (v3);
 			\draw (v2) to (v3);
 			\draw (v3) to (v4);
 			\draw (v3) to (v5);
			\draw (v3) to (v6);
			\draw (v5) to (v6);
		\end{tikzpicture}
		\\
		\\
		$G_1$ & $G_2$
	\end{tabular}
	\end{center}
	}
	\caption{A pair of $A_x^\mathbb{R}$-codeterminantal graphs that are not $A_x^\mathbb{Z}$-codeterminantal.
	}
	\label{Fig:CIAndIF}
\end{figure}
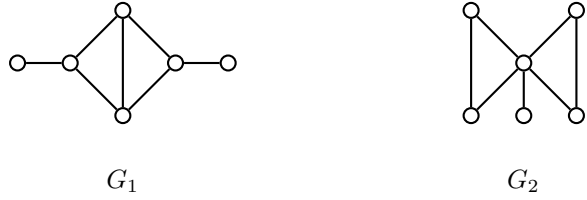

	These graphs are the unique pair of $A_x^{\mathbb{R}}$-codeterminantal graphs with $6$ vertices.
	\[\small
		I^\mathbb{R}_k\left(A_x(G_1)\right)=I^\mathbb{R}_k\left(A_x(G_2)\right)=
		\begin{cases}
		\langle 1\rangle & \text{ if } 1\leq k \leq 4,\\
		\langle x+1\rangle & \text{ if } k=5,\\
		\langle (x-1) (x+1)^2 (x^3-x^2-5x+1)\rangle & \text{ if } k=6,
		\end{cases}
	\]
	but when the base ring is $\mathbb{Z}[x]$, we observe that the characteristic ideals are different, and they are no longer codeterminantal:
	\[\small
		I^\mathbb{Z}_k\left(A_x(G_1)\right)=
		\begin{cases}
		\langle 1\rangle & \text{ if } 1\leq k \leq 4,\\
		\langle 2(x+1), (x+1) (x^2+1)\rangle & \text{ if } k=5,\\
		\langle (x-1) (x+1)^2 (x^3-x^2-5x+1)\rangle & \text{ if } k=6,
		\end{cases}
	\]
	and
	\[\small
		I^\mathbb{Z}_k\left(A_x(G_2)\right)=
		\begin{cases}
		\langle 1\rangle & \text{ if } 1\leq k \leq 3,\\
		\langle 2,(x+1)\rangle & \text{ if } k=4,\\
		\langle 4(x+1), (x+1) (x-3)\rangle & \text{ if } k=5,\\
		\langle (x-1) (x+1)^2 (x^3-x^2-5x+1)\rangle & \text{ if } k=6.
		\end{cases}
	\]
However, their varieties, in $\mathbb{R}$, for each $k$ are the same, that is
\[
V\left(I^\mathbb{Z}_k\left(A_x(G_1)\right)\right)=V\left(I^\mathbb{Z}_k\left(A_x(G_2)\right)\right),
\]
for each $k\in [n]$.
\end{example}

Actually, the last observation holds in general: if two graphs $G$ and $H$ are $M^\mathbb{R}_x$-codeterminantal, then it holds that
$$V\left(I^\mathbb{R}_k\left(M_x(G)\right)\right)=V\left(I^\mathbb{R}_k\left(M_x(H)\right)\right)=V\left(I^\mathbb{Z}_k\left(M_x(G)\right)\right)=V\left(I^\mathbb{Z}_k\left(M_x(H)\right)\right).$$ 
However, as shown in Example~\ref{example:differencechangingcoefficientspolynomialring}, the converse is not always true.

As aforementioned, not many relationships between the spectrum and the SNF are known. The next result contributes in this direction by presenting a necessary and sufficient condition for two graphs to be $M^\mathbb{R}_x$-codeterminantal.. 

\begin{theorem}
The graphs $G$ and $H$ are $M^\mathbb{R}_x$-codeterminantal graphs if and only if 
$V\left(I^\mathbb{Z}_k\left(M_x(G)\right)\right)=V\left(I^\mathbb{Z}_k\left(M_x(H)\right)\right)$.
\end{theorem}
\begin{proof}
If $G$ and $H$ are $M^\mathbb{R}_x$-codeterminantal, then their determinantal ideals are the same and thus $V\left(I^\mathbb{R}_k\left(M_x(G)\right)\right)=V\left(I^\mathbb{R}_k\left(M_x(H)\right)\right)$ for each $k$.
By Proposition~\ref{prop:equalvarietiesunivariate} we have $V\left(I^\mathbb{Z}_k\left(M_x(G)\right)\right)=V\left(I^\mathbb{R}_k\left(M_x(G)\right)\right)$ and $V\left(I^\mathbb{Z}_k\left(M_x(H)\right)\right)=V\left(I^\mathbb{R}_k\left(M_x(H)\right)\right)$ from the only if part follows.
The if part follows since $I^\mathbb{R}_k\left(M_x(G)\right)$ is principal and the generator is the polynomial $\prod_{\lambda\in V\left(I^\mathbb{Z}_k\left(M_x(G)\right)\right)}(x-\lambda)$.
\end{proof}

\begin{table}[h]
\begin{tabular}{cc|cc|cc|cc|cc}
$n$ & $N$    & $A^\mathbb{Q}_x$   & $A^\mathbb{Z}_x$ & $L^\mathbb{Q}_x$   & $L^\mathbb{Z}_x$ & $D^\mathbb{Q}_x$   & $D^\mathbb{Z}_x$ & $F^\mathbb{Q}_x$   & $F^\mathbb{Z}_x$ \\ \hline
5   & 21     & 0     & 0              & 0     & 0              & 0     & 0              & 0     & 0              \\
6   & 112    & 2     & 0              & 4     & 2              & 0     & 0              & 0     & 0              \\
7   & 853    & 63    & 6              & 115   & 14             & 22    & 0              & 43    & 8             \\
8   & 11117  & 1353  & 464            & 1611  & 280            & 658   & 186            & 745   & 130            \\
9   & 261080 & 46930 & 17894          & 40560 & 14935          & 25058 & ?
& 20455 & ?
\end{tabular}
\caption{Number of connected graphs with a $M_x^\mathcal{R}$-codeterminantal mate for various matrices over $\mathbb{Q}[x]$ and $\mathbb{Z}[x]$.
The number of vertices is denoted by $n$ and the number of connected graphs with $n$ vertices is denoted by $N$.}
\label{tab:statisticscodeterminantalideals}
\end{table}

Table \ref{tab:statisticscodeterminantalideals} shows the number of graphs with a $M_x^\mathcal{R}$-codeterminantal mate. In this work we focus in the cases when the ring $\mathcal{R}$ is either $\mathbb{Q}$ and $\mathbb{Z}$.
We observe that there are less $M_x^\mathbb{Z}$-codeterminantal graphs than $M_x^\mathbb{R}$-codeterminantal graphs.
In fact, it seems that univariate determinantal ideals with coefficients in $\mathbb{Z}[x]$ are the best algebraic invariant to distinguish graphs. 

It is worth to say that, aside that computing a determinantal ideal could take something between few seconds to 2 minutes, computing Table~\ref{tab:statisticscodeterminantalideals} has its difficulties, see Appendix~\ref{appendix:difficulties} for an account of some of them.

Now we recall the definition of cospectral graphs.

\begin{definition}
Let $M$ be either the adjacency, Laplacian, distance or distance Laplacian matrices.
Two graphs $G$ and $H$ are $M$-\emph{cospectral} if $M(G)$ and $M(H)$ have the same $M$-spectrum.
\end{definition}

We note that Table \ref{tab:statisticscodeterminantalideals} only calculates the number of codeterminental graphs through nine vertices, though cospectral graphs have been enumerated up to twelve vertices. This is due to the greatly increased computational complexity of considering all minors of the matrix.

\begin{table}[h]
	{
    \begin{tabular}{rccccc}
		\hline
        Number of vertices & 5 & 6 & 7 & 8 & 9 \\
        \hline
        Number of connected graphs & 21 & 112 & 853 & 11117 & 261080 \\
        \hline
        Adjacency & 0 & 2 & 63 & 1353 & 46930 \\
        Laplacian & 0 & 4 & 115 & 1611 & 40560 \\
        Distance & 0 & 0 & 22 & 658 & 25058 \\
        Distance Laplacian & 0 & 0 & 43 & 745 & 20455 \\

        \hline
	\end{tabular}
	}
\caption{Number of connected graphs with a cospectral mate for the adjacency, Laplacian, distance and the distance Laplacian matrices.}
	\label{Tab:cospectralgraphs_2}
\end{table}

Results in tables \ref{tab:statisticscodeterminantalideals} (for $\mathbb{Q}$) and \ref{Tab:cospectralgraphs_2} provide numerical evidence for Theorem \ref{theo:equivalencecospectralcodeterminantal}, which characterizes when $M_x^\mathbb{Q}$-codeterminental graphs are cospectral. In order to prove Theorem \ref{theo:equivalencecospectralcodeterminantal}, we will use the following result from
\cite{JN}.

\begin{theorem}\cite{JN}\label{theo:exitanceorthogonal}
If two graphs $G$ and $H$ are cospectral, then there exists an orthogonal matrix $U$ such that $M(G) = U^T M(H) U$.
\end{theorem}

\begin{theorem}\label{theo:equivalencecospectralcodeterminantal}
Two graphs $G$ and $H$ are $M$-cospectral if and only if the graphs are $M_x^\mathbb{Q}-codeterminantal$.
\end{theorem}

\begin{proof}
Let $G$ and $H$ be two $M$-cospectral graphs.
By Theorem \ref{theo:exitanceorthogonal} there exists an orthogonal matrix $U$ such that $M(G) = U^T M(H) U$.
Since $U^{-1}=U^T$, then applying Theorem~\ref{theo:equivalentcodeterminantal}, $M(G)$ and $M(H)$ are codeterminantal. The other direction is trivial.
\end{proof}

Table \ref{Tab:cospectralgraphs_2} provides the number of cospectral mates of a graph with respect of several associated matrices. In \cite{ah} it is reported that there are 19778 cospectral graphs with 9 vertices with respect to the distance Laplacian matrix. However, in our computation we obtain 20455 cospectral graphs (see Table~\ref{Tab:cospectralgraphs_2}), and this number coincides with the one in Table~\ref{tab:statisticscodeterminantalideals} (as expected by Theorem \ref{theo:equivalencecospectralcodeterminantal}). Hence,  we confirm that the number reported in~\cite{ah} is incorrect.



Next, we adopt the notation in \cite{vince} to introduce the definition of $M$-coinvariant graphs. 

\begin{definition}
Let $M$ be either  the adjacency, Laplacian, distance or distance Laplacian matrices.
Two graphs $G$ and $H$ are $M$-{\it coinvariant} if the SNFs computed over $\mathbb{Z}$ of $M(G)$ and $M(H)$ are the same.
\end{definition}

If we wish to compute the generator of the $k$-th determinantal ideal $I_k^\mathbb{Z}(M)$ of an integer matrix $M$ (without indeterminate), we can just apply Corollary~\ref{theo:determinatalevaluatedsingleind}. 
Thus if $f_1 \ldots f_n$ are the elements in the diagonal of the SNF, then the generator of $I_k^\mathbb{Z}(M)$ is equal to $\prod_{j=1}^k f_j$, which coincides with $\Delta_k=\gcd(\minors_k(M))$. This avoids computing all $k$-minors. 
Based on this, we conclude that coinvariant coincides with $M^\mathbb{Z}$-codeterminantal. 

\begin{theorem}\label{theo:equivalencecoinvariantcodeterminantal}
Two graphs $G$ and $H$ are $M$-coinvariant if and only if the graphs are $M^\mathbb{Z}-codeterminantal$.
\end{theorem}

Note that two graphs $G$ and $H$ being $M$-coinvariant implies that the cokernel of $M(G)$ and $M(H)$ are isomorphic, in particular the torsion part of the cokernel are also isomorphic.
The converse is not always true, since $M$-coinvariant constraints the number of vertices of the graphs to be the same.
An interesting example comes from the Laplacian matrix, where it holds that if $H$ is a dual of a planar graph $G$, then the critical groups of $G$ and $H$ are isomorphic \cite{cori,vince}.
The computation of the invariant factors of the Laplacian matrix is an important technique used for the understanding of the critical group and the graph properties.
The number of zeros in the diagonal of the SNF of the Laplacian matrix gives the number of connected components, meanwhile the multiplication of the invariant factors gives the number of spanning trees.

Several researchers have addressed the question of how often the critical group is cyclic. 
In \cite{lorenzini2008} and \cite{wagner} Lorenzini and Wagner, based on numerical data, suggest that we could expect to find a substantial proportion of graphs having a cyclic critical group.
Based on this, Wagner conjectured \cite{wagner} that almost every connected simple graph has a cyclic critical group.
A recent study \cite{wood} concluded that the probability that the critical group of a random graph is cyclic is asymptotically at most
\[
	\zeta(3)^{-1}\zeta(5)^{-1}\zeta(7)^{-1}\zeta(9)^{-1}\zeta(11)^{-1}\cdots \approx 0.7935212,
\]
where $\zeta$ is the Riemann zeta function; differing from Wagner's conjecture. 

\begin{table}[h]
	{
    \begin{tabular}{rcccccc}
		\hline
        Number of vertices & 4 & 5 & 6 & 7 & 8 & 9 \\
        \hline
        Number of connected graphs & 6 & 21 & 112 & 853 & 11117 & 261080 \\
        \hline
        Adjacency          & 4 & 20 & 112 & 853 & 11117 & 261061 \\
        Laplacian          & 2 &  8 &  57 & 526 &  8027 & 221830  \\
        Distance           & 2 & 15 & 102 & 835 & 11080 & 260771   \\
        Distance Laplacian & 0 &  0 &   0 &  18 &   455 & 16501  \\

        \hline
	\end{tabular}
	}
\caption{Number of connected graphs with a coinvariant mate for the adjacency, Laplacian, distance and the distance Laplacian matrices.}
	\label{Tab:coinvariantgraphs}
\end{table}

Table \ref{Tab:coinvariantgraphs} provides the number of $M$-coinvariant mates with the same number of vertices.
Biggs suggested in \cite{biggs1999} that the SNF can be used to distinguish graphs in cases where other algebraic invariants, such as those derived from the spectrum, fail.
Looking at Table \ref{Tab:coinvariantgraphs}, we observe that the SNF of the adjacency, Laplacian and distance matrices may not be good in distinguishing graphs, since for these matrices almost all graphs with at most 9 vertices have a coinvariant mate. On the other hand, for the distance Laplacian seems there is more hope for such characterization. Comparing the values from Table~\ref{Tab:cospectralgraphs_2}, we observe that the SNF of the distance Laplacian seems to perform better for distinguishing graphs than its spectrum.

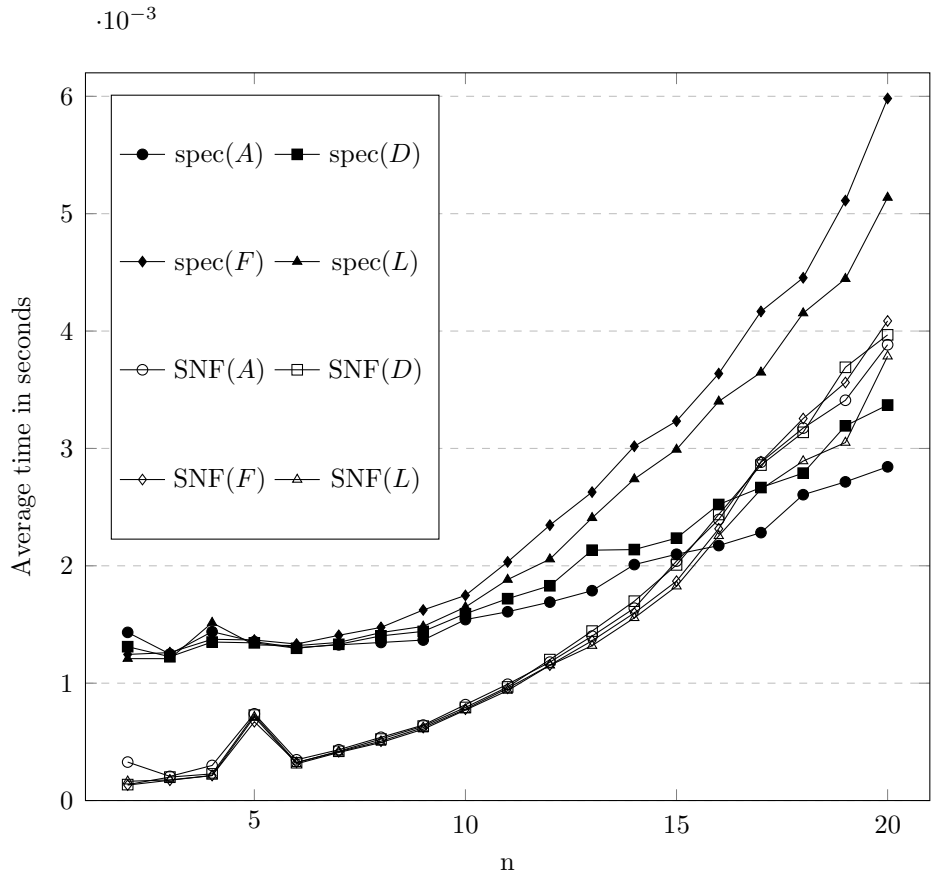
\begin{figure}[ht]
    \centering
    \begin{tikzpicture}[trim axis left]
\begin{axis}[
    scale only axis,
    title={},
    xlabel={n},
    ylabel={Average time in seconds},
    y label style={at={(axis description cs:0.1,.5)},anchor=south},
    width=\textwidth-0.9cm, 
    xmin=1, xmax=21,
    ymin=0.0, ymax=0.0062,
    xtick={5,10,15,20},
    legend pos=north west,
    ymajorgrids=true,
    grid style=dashed,
    legend columns=2
]
 
\addplot[
    mark=*,
    legend entry={ $\spec(A)$} 
    ]
    coordinates {
    (2,0.001432896)(3,0.001248002)(4,0.001438459)(5,0.001353673)(6,0.001302683)(7,0.001326613)(8,0.001346801)(9,0.00136734)(10,0.001541712)(11,0.001608491)(12,0.001690675)(13,0.001787899)(14,0.002010369)(15,0.002098544)(16,0.002172303)(17,0.002282576)(18,0.002605789)(19,0.00271514)(20,0.002843288)
    };
 
\addplot[
    mark=square*,
    legend entry={ $\spec(D)$} 
    ]
    coordinates {
    (2,0.001310825)(3,0.001225114)(4,0.001350443)(5,0.001343614)(6,0.001297551)(7,0.001331156)(8,0.00140287)(9,0.001441231)(10,0.001591571)(11,0.001719832)(12,0.001829848)(13,0.002133098)(14,0.002138368)(15,0.002235616)(16,0.002522465)(17,0.002666772)(18,0.002789515)(19,0.003192448)(20,0.003369964)
    };

\addplot[
    mark=diamond*,
    legend entry={ $\spec(F)$} 
    ]
    coordinates {
    (2,0.001245022)(3,0.001260519)(4,0.001374046)(5,0.001367898)(6,0.001333422)(7,0.001407776)(8,0.001475205)(9,0.001623615)(10,0.001747802)(11,0.002033114)(12,0.002346251)(13,0.002627918)(14,0.003018506)(15,0.003232839)(16,0.003637701)(17,0.004167122)(18,0.004453677)(19,0.005111485)(20,0.00598164)
    };
    
\addplot [
    mark=triangle*,
    legend entry={ $\spec(L)$} 
    ]
    coordinates {
    (2,0.001209021)(3,0.001208901)(4,0.00151515)(5,0.001326164)(6,0.001320056)(7,0.00134733)(8,0.001431079)(9,0.001483793)(10,0.001648373)(11,0.001882418)(12,0.002056349)(13,0.002408438)(14,0.002739031)(15,0.002990731)(16,0.003400665)(17,0.003647823)(18,0.004152752)(19,0.004445181)(20,0.005137276)
    };

\addplot [
    mark=o,
    legend entry={ $\SNF(A)$} 
    ]
    coordinates {
    (2,0.000327826)(3,0.000205994)(4,0.000299017)(5,0.000739018)(6,0.000347208)(7,0.00043298)(8,0.000538901)(9,0.000643009)(10,0.00081702)(11,0.000990981)(12,0.001179206)(13,0.001401)(14,0.00163887)(15,0.002040055)(16,0.00239467)(17,0.002882401)(18,0.00317337)(19,0.00341027)(20,0.003883558)
    };
\addplot [
    mark=square,
    legend entry={ $\SNF(D)$} 
    ]
    coordinates {
    (2,0.000135899)(3,0.000199914)(4,0.000227133)(5,0.000728062)(6,0.000322295)(7,0.000419952)(8,0.000523954)(9,0.000634249)(10,0.000793129)(11,0.000967963)(12,0.001200774)(13,0.001443802)(14,0.001698554)(15,0.002007575)(16,0.002437359)(17,0.002858075)(18,0.003137355)(19,0.0036912)(20,0.00396753)
    };
\addplot [
    mark=diamond,
    legend entry={ $\SNF(F)$} 
    ]
    coordinates {
    (2,0.000131845)(3,0.000172973)(4,0.000212987)(5,0.000671364)(6,0.000325194)(7,0.000415596)(8,0.000505762)(9,0.000621488)(10,0.000778582)(11,0.000956125)(12,0.001153547)(13,0.001361445)(14,0.001597713)(15,0.001869193)(16,0.002315442)(17,0.002886916)(18,0.003254736)(19,0.003561589)(20,0.004086207)
    };
\addplot [
    mark=triangle,
    legend entry={ $\SNF(L)$} 
    ]
    coordinates {
    (2,0.000164032)(3,0.000176072)(4,0.000213265)(5,0.00071303)(6,0.000312473)(7,0.000409506)(8,0.00049432)(9,0.000607779)(10,0.00077111)(11,0.00093707)(12,0.001152352)(13,0.001319503)(14,0.001556238)(15,0.00182564)(16,0.002255071)(17,0.002651947)(18,0.002892255)(19,0.003049804)(20,0.003784927)
    };
 
\end{axis}
\end{tikzpicture}
    \caption{The average time of computing the spectrum and the Smith normal form of several matrices associated to graphs with $n$ vertices.}
    \label{fig:averagetime}
\end{figure}

It is well known that computing the characteristic polynomial or the SNF of an integer matrix can be done in polynomial time.
We carried out a performance analysis, as it is shown in Figure~\ref{fig:averagetime}, where we report the average time taken in the computation of these two properties over matrices of randomly generated connected graphs.
For every generated graph we used four of its associated matrices, the Adjacency (A), Laplacian (L), Distance (D) and Distance Laplacian (F) matrices.
In Figure \ref{fig:averagetime}, filled markers are for the characteristic polynomial and empty markers are for the SNF.
Every point in the plot represents the average time of a subset of connected graphs on $n$ vertices. 

We considered all the connected graphs up to $9$ vertices. From $n=10$ vertices and above, the number of connected graphs is substantially large, hence we developed a model to randomly generate samples of the graphs. 
A random connected graph can be generated as follows: given $n$ vertices we generate a random spanning tree and then append each of the other possible edges with probability $p$. 
A random graph depends on a parameter $p$ other than the number of vertices, which measures the density of the graph ($p=0$ for a minimal connected graph and $p=1$ for a complete graph). 
The distribution of the number of edges of all the connected graphs on $n$ vertices behave similar to a binomial distribution $Bin(m, 0.5)$, where $m$ is the number of edges of the complete graph on $n$ vertices.
To replicate this behavior, the density of each random graph on $n$ vertices was $d=Bin(m, 0.5) / m$.
The sizes of all our samples are $185,656$ for $n=10,\ldots,20$.

The software used to make the above computations is Python 2.7.15 bundled with Sage 8.8 on a Windows 10 Pro (64-bit), Intel(R) i5-3210M at 2.5 GHz (4 CPUs) and 8 GB RAM machine.

In Figure~\ref{fig:averagetime} we observe that for small $n$ there is a clear advantage in computing the SNF than the spectrum.
However, this tendency is no longer true after $n=15$.
At $n=20$, the worst performance is displayed by $\spec(F)$, meanwhile the best performance is shown by $\spec(A)$.
Not quite far from $\spec(A)$ are $\spec(D)$, $\SNF(A)$, $\SNF(L)$, $\SNF(D)$ and $\SNF(F)$. 
In general, the computation of the SNF of all matrices seems to behave similarly, meanwhile there is a clear difference with the different spectra ($\spec(A)$ and $\spec(D)$ perform better than $\spec(F)$ and $\spec(D)$).

\section{Graphs determined by their determinantal ideals}\label{sec:graphsdeterminedbytheirdeterminantalideals}

Since its introduction by Aouchiche and Hansen in \cite{ah2013}, the distance Laplacian matrix has received quite some attention regarding its spectral properties, see for instance \cite{ah2014,bdhlrsy2019,DAH2018,np2014}. Although from Tables \ref{tab:statisticscodeterminantalideals} and \ref{Tab:coinvariantgraphs} from the previous section we observe that the distance Laplacian matrix provides the best graph invariants in order to characterize graphs, there are not yet many known results on spectral characterizations of graphs using the distance Laplacian matrix of a graph \cite{ah,np2014}. In this section we study families of graphs that are determined by the SNF of the distance Laplacian matrix.  

\begin{theorem}\label{teo:completegraphsaredeterminedbyDL}
Complete graphs are determined by the SNF of the distance Laplacian matrix.
\end{theorem}

 Before proving Theorem~\ref{teo:completegraphsaredeterminedbyDL}, we need the following results.

\begin{theorem}\cite{at}\label{teo:clasificationofgraphswith1trivialdistance}
A connected graph has only one trivial distance ideal over $\mathbb{Z}[X]$ if and only if $G$ is either a complete graph or a complete bipartite graph.
\end{theorem}

 The next result is a consequence  of Proposition~\ref{prop:evalmultiplevariables} and Theorem~\ref{teo:clasificationofgraphswith1trivialdistance}.

\begin{corollary}
Let $G$ be a connected graph such that
its distance Laplacian matrix has at most one invariant factor equal to $1$, then $G$ is a complete graph or a complete bipartite graph.
\end{corollary}
\begin{proof}
By Proposition~\ref{prop:evalmultiplevariables}, if $G$ is a graph whose $I_k^\mathbb{Z}(D_X(G))$ is trivial, then, after evaluating $X=\deg(G)$, we have $\Delta_k(D(G))=1$. This implies that the family of graphs whose distance Laplacian matrix has at most one invariant factor  equal to 1 is contained in the family of graphs with at most one trivial distance ideals over $\mathbb{Z}[X]$. Hence, the result follows from Theorem~\ref{teo:clasificationofgraphswith1trivialdistance}.
\end{proof}

Note that the number of vertices of $G$ can be deduced from the SNF of the distance Laplacian matrix by looking at the number of elements in the diagonal of SNF of the distance Laplacian matrix.

Recall also that the distance Laplacian matrix of a complete graph coincides with its Laplacian matrix.
The SNF of the Laplacian matrix of a complete graph is known to be $\diag(1,n-1,\dots,n-1,0)$.
This implies that the sandpile group for the complete graph $K_n$ is isomorphic to $\mathbb{Z}_{n-1}^{n-2}$.

Now we are ready to prove Theorem~\ref{teo:completegraphsaredeterminedbyDL}.
\begin{proof}[Proof of Theorem~\ref{teo:completegraphsaredeterminedbyDL}]
We know that the second invariant factor of the SNF of $F(K_n)$ and the gcd of the 2-minors of $F(K_n)$ are equal to the number of vertices minus one. Next, we shall prove that in fact complete graphs are the unique graphs with this property.
Assume $n\geq m\geq 1$.
And let
$$
F(K_{n,m})=
\begin{bmatrix}
(2n+m)I_n-2J_n & -J_{n,m} \\
-J_{m,n} & (n+2m)I_m-2J_m\\
\end{bmatrix}
$$
be the distance Laplacian matrix of the complete bipartite graph $K_{n,m}$, where $J_n$ denote the all-one matrix of size $n\times n$.
It follows that $\Delta_1(F(K_{n,m}))=1$.
Let us consider two cases: when $K_{m,n}$ is a star and when it is not a star.
In the first case, the 2-minors (with positive leading coefficient) of $F(K_{n,1})$ are:
\[
L_1=\{ 4n^2 - 4n - 3, 2n + 1, 2n^2 - n - 1\}.
\]
Considering $n$ and $m$ as indeterminates, the Gr\"obner basis of the ideal $\langle L_1\rangle\subseteq \mathbb{Z}[n,m]$ is generated by $2n+1$, which is different from $n=(n+1)-1$.
Now, the 2-minors (with positive leading coefficient) of $F(K_{n,m})$ are:
\begin{eqnarray*}
L_2 & = & \{ 
4n^2 + 4nm - 8n + m^2 - 4m, 
2n + m, 
0, 
2n^2 + 5nm - 6n + 2m^2 - 6m + 3, \\
& &
4n + 2m - 3, 
2n + 4m - 3, 
3, 
n + 2m, 
n^2 + 4nm - 4n + 4m^2 - 8m \}.
\end{eqnarray*}
Considering $n$ and $m$ as indeterminates, the Gr\"obner basis of the ideal $\langle L_2\rangle\subseteq \mathbb{Z}[n,m]$ is generated by $n+2m$ and 3, which is not principal.
In fact, $m+n-1\notin \langle n+2m,3\rangle\subseteq\mathbb{Z}[n,m]$.
From which follows that complete graphs are the only graphs having the first invariant factor of the SNF of the distance Laplacian matrix equal to 1 and the second invariant factor equal to the number of vertices minus one.
\end{proof}

\begin{theorem}\label{teo:stargraphsaredeterminedbyDL}
Star graphs are determined by the SNF of the distance Laplacian matrix.
\end{theorem}
\begin{proof}
From the previous proof of Theorem~\ref{teo:completegraphsaredeterminedbyDL}, we know that the first invariant factor of the SNF of $F(K_{n,1})$ is equal to 1 and the second invariant factor is equal to $2n+1$. 
Consider $K_{n,m}$ with $n\geq m\geq 2$.
We have that $2(n+m)+1\notin\langle n+2m,3 \rangle\subseteq \mathbb{Z}[n,m]$, from which follows that star graphs are the only graphs having the first invariant factor equal to 1 and the second invariant factor equal to two times the number of vertices plus one.
\end{proof}

Recall that Lorenzini and Vince \cite{lorenzini1991,vince} showed that complete graphs are the only graphs with only one invariant factor equal to one in the SNF of the Laplacian matrix, which leads directly to the following result.

\begin{theorem}
Complete graphs are determined by the SNF of the Laplacian matrix.
\end{theorem}



As we have seen the best determinantal ideals, aside to the critical and distance ideals, to distinguish graphs are the univariate determinantal ideals in $\mathbb{Z}[x]$ since they encode information on the SNF and the spectrum. 
They are more difficult to compute since these ideals are not principal in general, and we have to compute their Gr\"obner bases instead.
However, the last univariate determinantal ideal of a $n\times n$ matrix $M$, $I_n^\mathbb{Z}(M_x)$, is generated by the determinant $\det(M_x)$, from which follows that if a graph $G$ is determined by the spectrum, then $G$ is determined by their univariate determinantal ideals. 
Thus univariate determinantal ideals in $\mathbb{Z}[x]$ can be used to distinguish graphs in cases where the spectrum fail, as in Example \ref{example:differencechangingcoefficientspolynomialring}.

\subsection*{Acknowledgments}
Carlos A. Alfaro and Marcos C. Vargas are partially supported by CONACyT and SNI.


\newpage

\appendix
\section{Computing characteristic and distance characteristic ideals with Macaulay2}
In this appendix a code for computing the characteristic ideals of graphs with Macaulay2 \cite{m2} is provided.
For this example, we consider the graph $\ltimes$, and the polynomial ring $\mathcal{R}=\mathbb{Z}[t]$.
We define the determinantal ideal $I$ generated by the set of minors of size $i$ of matrix $M$ with the code {\tt minors(i,M)}, and compute its Gr\"obner bases with {\tt gens gb I}.
Thus, the following will compute the characteristic ideals of $\ltimes$ over $\mathbb{Z}$.
\begin{lstlisting}
G = graph({{0,1},{0,2},{0,3},{0,4},{2,3}})
R = ZZ[t]
M = diagonalMatrix{t,t,t,t,t}-G.adjacencyMatrix
for i from 1 to 5 do (
   I = minors(i,M);
   print(gens gb I);
)
\end{lstlisting}
The output is the following:
\begin{lstlisting}[numbers=none]
| 1 |
| 1 |
| 1 |
| 2 t+1 |
| t5-5t3-2t2+2t |
\end{lstlisting}
from which follows that $\gamma_\mathbb{Z}(G)=3$ and $I_4^{\mathbb{Z}}(G,t)=\langle2,t+1\rangle$ and $I_4^{\mathbb{Z}}(G,t)=\langle t^5-5t^3-2t^2+2t\rangle$.

Computing Gr\"obner basis of the distance characteristic ideals is also easy with Macaulay2:

\begin{lstlisting}
D = diagonalMatrix{t,t,t,t,t}-distanceMatrix G
for i from 1 to 5 do (
   I = minors(i,D);
   print(gens gb I);
)
\end{lstlisting}
The output is the following.
\begin{lstlisting}[numbers=none]
| 1 |
| 1 |
| 1 |
| 6 t-1 |
| t5-25t3-70t2-66t-20 |
\end{lstlisting}

We found useful to use McKay's Nauty software \cite {mckay} as Macaulay2's package {\tt NautyGraphs}.
For example the graph $\ltimes$ can also be loaded with the following code:
\begin{lstlisting}
loadPackage "NautyGraphs"
stringToGraph "Dt_"
\end{lstlisting}

\section{Difficulties in computing determinantal ideals}\label{appendix:difficulties}
Many mathematicians have faced the difficulty of either to trust or to not the results obtained under computer algebra systems  \cite{dpv}.
In computing Table \ref{tab:statisticscodeterminantalideals}, we had a similar situation when we had to compare the equality of two ideals, which is a common problem in commutative algebra.
Consider the following code in Cocalc, in which, we compare two ideals by two methods: one direct (line 4) and the other by using Gr\"obner bases (line 7).
\begin{lstlisting}
R = PolynomialRing(ZZ, 'x', implementation="singular")
I = R.ideal([x^3 + 1086*x^2 - 22022*x + 108388, 1106*x^2 - 22120*x + 108388])
J = R.ideal([x^3 - 20*x^2 + 98*x, 1106*x^2 - 22120*x + 108388])
I == J
I = I.groebner_basis()
J = J.groebner_basis()
I == J
\end{lstlisting}
The output of this code is the following.
\begin{lstlisting}[numbers=none]
True
False
\end{lstlisting}
The first one says that the ideals are the same and the other that they are not.
Therefore, one is incorrect.
For this example, the reader can easily verify by hand that both ideals are equal.
The same error occours if we try to use Maculay2 on Cocalc.
\begin{lstlisting}
P = macaulay2.ring('ZZ', '[x]')
I = macaulay2.ideal( ("x^3 + 1086*x^2 - 22022*x + 108388", "1106*x^2 - 22120*x + 108388") )
J = macaulay2.ideal( ("x^3 - 20*x^2 + 98*x", "1106*x^2 - 22120*x + 108388") )
I == J
I = I.gb()
J = J.gb()
I == J
\end{lstlisting}
But when the computation is done in pure Macaulay2, the result is correct.
\begin{lstlisting}
R = ZZ[x];
I = ideal({x^3 + 1086*x^2 - 22022*x + 108388,1106*x^2 - 22120*x + 108388});
J = ideal({x^3 - 20*x^2 + 98*x,1106*x^2 - 22120*x + 108388});
I == J
gens gb I == gens gb J
\end{lstlisting}
This mistake could be due to the computation of the Gr\"obner bases of polynomials with coefficients in $\mathbb{Z}$ is not a common task, making it susceptible.
\end{document}